\documentclass{amsart}
\usepackage[latin1]{inputenc}
\usepackage[english]{babel}
\usepackage{amstext}
\usepackage{latexsym}
\usepackage{amsmath}
\usepackage{amssymb}
\usepackage{amsfonts}
\usepackage{graphicx}
\usepackage{amssymb}
\usepackage{amsthm}
\usepackage{epsfig}
\usepackage{tikz}
\usetikzlibrary{matrix,arrows,decorations.pathmorphing}
\usepackage{graphicx}
\usepackage{enumitem}
\usepackage{xcolor}
\allowdisplaybreaks

\input xy
\xyoption{all}
\newtheorem{defn}{Definition}[section]

\newtheorem{lemma}[defn]{Lemma}

\newtheorem{prop}[defn]{Proposition}

\newtheorem*{theoA}{Theorem A}

\newcommand{\Z}{\mathbb Z}

\begin{document}

\title[Density of non-zero exponent of contraction]{Density of non-zero exponent of contraction for pinching cocycles in $\text{Hom}(S^1)$}

\author{Catalina Freijo}
\address{Faculdade de Ciencias, Universidade de Lisboa, Campo Grande 016, 1749-016 Lisboa, Portugal.}
\email{cfreijo@fc.ul.pt}

\author{Karina Marin}
\address{Departamento de Matem\'atica, Universidade Federal de Minas Gerais, Av. Ant\^onio Carlos 6627, 31270-901 Belo Horizonte
Minas Gerais, Brazil}
\email{kmarin@mat.ufmg.br}

\begin{abstract}
We consider pinching cocycles taking values in the space of homeomorphisms of the circle over an hyperbolic base. Using the Invariance Principle of Malicet, we prove that the cocycles having non-zero exponents of contraction are dense. In this article we generalize some common notions an results known of linear cocycles and cocycles of diffeomorphisms, to the non-linear non-differentiable case.  
\end{abstract}


\maketitle

\section{Introduction}

A continuous cocycle over a transformation $f\colon X\to X$ is a map $F\colon\mathcal{E} \to\mathcal{E}$, where $\mathcal{E}$ is a fiber bundle which fibers $N$ are topological spaces, such that the following diagram 
$$\begin{matrix}
&\mathcal{E}& \xrightarrow{F} &\mathcal{E} \\
&\downarrow & &\downarrow\\
&X & \xrightarrow{f} &X
\end{matrix}$$
commutes and the action on the fibers $F_x\colon \mathcal{E}_x\to\mathcal{E}_{f(x)}$ is an homeomorphism. In this case its orbit takes the form
$$F^n_x=F_{f^{n-1}(x)}\circ\ldots \circ F_{x}.$$

In this paper, we consider the particular case when the fiber $N=S^1$ and the maps $F_x$ are bi-H\"older homeomorphisms. In this context we prove, that under certain conditions, the hyperbolicity of the base map $f$ induces contraction properties in the fibers. This generalizes known results of linear cocycles and cocycles of diffeomorphisms.  

When studying linear cocycles, we consider the trivial bundle $X\times\mathbb{R}^d$ and a cocycle defined by $F_x$ being linear maps acting on $\mathbb{R}^d$. In this case, the Lyapunov exponents, $$\lim_{n\to +\infty} \frac{1}{n}\log \|F^n_xv\|, \ v\in \mathbb{R}^d,$$ provides information about the asymptotic behavior of the dynamics on the fibers. 

Furstenberg proved in \cite{F} that for random product of matrices the case of non-zero Lyapunov exponents is open and dense. This result was extended to linear cocycles over uniform hyperbolic maps by \cite{BGV} and when the base is non-uniformly hyperbolic by \cite{V}. 


For cocycles of diffeomorphisms, that is, when the fiber $N$ is a manifold and the maps $F_x$ are diffeomorphisms, we still can obtain information of the dynamic through the Lyapunov exponents, $$\lim_{n\to +\infty} \frac{1}{n}\log \|D_vF_x^n\xi\|,\; \xi\in T_v\mathcal{E}_x.$$ In this context, characterization of the set of cocycles with non-zero Lyapunov exponents has been studied in \cite{AV}.

The principal tool that is used for understanding the structure of cocycles with zero Lyapunov exponents is the Invariance Principle. The Invariance Principle was originally proved by Furstenberg \cite {F} and Ledrappier \cite{L} in the linear case and adapted to the non-linear differentiable context by  Avila and Viana in \cite{AV}. This result states that if the Lyapunov exponents vanish, then the fibers carries some structure that remains invariant by a family of homeomorphisms acting between the fibers. The center question is whether by perturbation of the cocycle this structure can be broken.

Malicet introduced in \cite{M} a version of the Invariance Principle for continuous cocycles of circle homeomorphisms. This result uses the notion of exponent of contraction, $$\limsup_{  {q}\to   {p}}\limsup_{n\to +\infty}\frac{\log(d(  {f}^n_{  {x}}({  {p}}),  {f}^n_{  {x}}({  {q}})))}{n}, \; p\in S^1,$$ which generalizes the concept of Lyapunov exponent and provides information about the dynamic on the fibers. 

The Invariance Principle of Malicet was used in \cite{M} to conclude several results about random walks of homemorphisms of the circle and has been applied by other authors to different settings, see for example \cite{CS} and \cite{DGR}.

In the present work, we use the Invariance Principle of Malicet to study cocyles whose action on the fibers are circle homeomorphisms and extend the known results of non-zero Lyapunov exponents to the non-linear non-differentiable case using the notion of exponent of contraction.


\section{Preliminaries and statements}

Let ${\Omega}\subset \{1,...,k\}^{\mathbb{Z}}$ be a sub-shift of finite type and ${\sigma}\colon  {\Omega}\to  {\Omega}$ denote the left-shift map defined by $  {\sigma}(  {x}_n)_{n\in\mathbb{Z}}=(  {x}_{n+1})_{n\in\Z}$. 

For every $\rho>1$, we can define a distance in $  {\Omega}$ by $d_{\rho}(  {x},  {y})=\rho^{-N_{  {x},  {y}}}$, where $N_{  {x},  {y}}=\max\{N\geq 0; x_n=y_n \text{ for every } |n|<N\}$. Since the topologies given by the different constants $\rho$ are equivalent, from now on we consider $\rho$ fixed and denote this distance as $d_{\Omega}$. 

Let $P^s\colon \Omega\to \Omega^{+}$ be the projection onto the positive coordinates and $P^u\colon \Omega\to \Omega^{-}$ the projection onto the negative coordinates. 

For every $i\in \{1,...,k\}$, denote $[0;i]=\{x\in \Omega: x_0=i\}$ and $\psi_i$ the homeomorphism $$\psi_i\colon P^u([0;i])\times P^s([0;i])\to [0;i].$$

\begin{defn} Given a ${\sigma}$-invariant measure $\mu$, define $\mu^s=P^s_*\mu$ and $\mu^s=P^s_*\mu$. 

The measure $\mu$ is said to have \emph{local product structure} if there exists a continuous function $\rho\colon \Omega\to (0,\infty)$ such that for every $i\in \{1,..., k\}$ and every measurable set $E\subset [0;i]$ we have $$\mu(E)=\int (\chi_E\circ \psi_i)\,\rho\, d\mu^u\times d\mu^s.$$ 
\end{defn}

It has been shown in \cite{H} and \cite{Ll}, that in the setting of this paper, equilibrium states of H\"older potentials have local product structure. 

Let $\text{Hom}(S^1)$ be the set of homeomorphisms of the circle and $\mathcal{H}_{\beta}(S^1)\subset \text{Hom}(S^1)$ be the set of $\beta$-H\"older maps which inverse is also $\beta$-H\"older. 

We consider $\mathcal{H}_{\beta}(S^1)$ with the metric $$d_{\max}(  {f}_1,  {f}_2)=\max\{d_{\beta}(  {f}_1,  {f}_2), d_{\beta}(  {f}_1^{-1},  {f}_2^{-1})\},$$ where $d_{\beta}$ is the usual H\"older distance. This means, $$d_{\beta}(  {f}_1,  {f}_2)=\sup_{  {p}\in S^1} d(  {f}_1(  {p}),  {f}_2(  {p}))+ |H_{\beta}(  {f}_1)-H_{\beta}(  {f}_2)|,$$ where ${H}_{\beta}(  {f})$ denote the H\"older constant of $f$ and $d$ is the standard distance on $S^1$.

Let $\mathcal{H}_{\alpha}(  {\Omega},\mathcal{H}_{\beta}(S^1))$ be the set of $\alpha-$H\"older maps defined from $(  {\Omega},d_{\Omega})$ to \break $(\mathcal{H}_{\beta}(S^1),d_{\max})$ endowed with the usual H\"older distance $d_{\alpha}$. 

The cocycle induced by $  {f}\in\mathcal{H}_{\alpha}(  {\Omega},\mathcal{H}_{\beta}(S^1))$ is the skew-product $  {F}_{  {f}}\colon   {\Omega}\times S^1\to  {\Omega}\times S^1$ defined by $$  {F}_{  {f}}(  {x},  {p})=(  {\sigma}(  {x}),   {f}_{  {x}}(  {p})).$$ For the rest of the work we use the notation $  {F}$ when there is not needed to specify the map $  {f}$.

Using the notation 
$$  {f}^n_{  {x}}=  {f}_{  {\sigma}^{n-1}(  {x})}\circ\ldots\circ  {f}_{  {x}}, \ \ \forall x\in\Omega,$$
the iterates of $  {F}$ are given by $  {F}^n(  {x},  {p})=(  {\sigma}^n(  {x}),  {f}^n_{  {x}}(  {p}))$.

In this context we introduce the notion of exponent of contraction due to Malicet \cite{M}. This quantity measures the contracting exponential rate of the action on the fibers. 

\begin{defn}
The exponent of contraction of $  {F}$ at the point $(  {x},  {p})$ is the non positive quantity 
$$\lambda_{con}(  {F},  {x},  {p})=\limsup_{  {q}\to   {p}}\limsup_{n\to +\infty}\frac{\log(d(  {f}^n_{  {x}}({  {p}}),  {f}^n_{  {x}}({  {q}})))}{n}.$$
If $  {m}$ is a $  {F}$- invariant probability measure, the exponent of contraction of $  {m}$ is defined as 
$$\lambda_{con}(  {F},  {m})=\int_{  {\Omega}\times S^1}\lambda_{con}(  {F},  {x},  {p})d  {m}(  {x},  {p}).$$
\end{defn}
Note that $\lambda_{con}$ is $  {F}$-invariant, then $\lambda_{con}$ is constant $  {m}$-almost everywhere if $  {m}$ is ergodic.

\begin{theoA} Let $f\in\mathcal{H}_{\alpha}(  {\Omega},\mathcal{H}_{\beta}(S^1))$ and ${\mu}$ be a $  {\sigma}$-invariant measure satisfying the following assumptions: 
\begin{itemize}
\item [\emph{(a)}] $\mu$ is ergodic, fully supported and has local product structure.
\item [\emph{(b)}] $  {F}_f$ is a \emph{pinching cocycle}, that is there exists a periodic point of $  {\sigma}$, $  {x}_0$,  such that $  {f}^{\text{per}(  {x}_0)}_{  {x}_0}$ has two fixed points: one attractor and one repelling. 
\item[\emph{(c)}]$  {F}_f$ is \emph{su-dominated}, this means that there exists a constant $c<1$ such that
$$\label{dom}\max\{H_{\alpha}(  {f}_*)\rho^{-\alpha\beta}: \,   {f}_*\text{ is a fiber map of }  {F}_f \text{ or }  {{F}_f}^{-1}\}\leq c.$$
\end{itemize}
Then for every open neighborhood $\mathcal{U}$ of $  {f}$ in $\mathcal{H}_{\alpha}(  {\Omega},\mathcal{H}_{\beta}(S^1))$, there exists an open set $\mathcal{V} \subset \mathcal{U}$ such that for any $  {g}\in  \mathcal{V}$ and any $  {F}_{  {g}}$-invariant measure $  {m}$ projecting to $  {\mu}$, we have
$$\lambda_{con}(  {F}_{  {g}},  {m})<0 \text{ or } \lambda_{con}( {{F}_  {g}}^{-1},  {m})<0.$$
\end{theoA}

\section{Proof of Theorem A} 

For simplifying the notation we develop the proof assuming $\alpha=\beta=1$, that is, we suppose the maps are Lipschitz continuous. However, the proof still works without this consideration. 

From now on we consider $f\in \mathcal{H}_{1}(    {\Omega},\text{Lip}(S^1))$ and ${\mu}$ satisfying the hypotheses of Theorem A.

We say the cocycle $  {F}_f$ has stable holonomies if there exists a collection of $\gamma$-H\"older ho\-meo\-mor\-phisms $h^{s}_{  {x},  {y}}\colon {S}^1\to {S}^1 $, with uniform H\"older constant, defined for every $  {y}\in W^s_{loc}(  {x})$ satisfying 
\vspace{0.2cm}

\begin{enumerate}[label=\emph{(\alph*)}]
		\item $h^{s}_{    {y},    {z}}\circ h^{s}_{    {x},    {y}}= h^{s}_{    {x},    {z}}$ and $h^{s}_{    {x},    {x}}=Id$;
		\vspace{0.2cm}
		
		\item $h^{s}_{    {\sigma}(    {x}),    {\sigma}(    {y})}=    {f}_{    {y}}\circ h^{s}_{    {x},    {y}}\circ     {f}_{    {x}}^{-1}$; 
		\vspace{0.2cm}
		
		\item $(    {x},    {y})\mapsto h^{s}_{    {x},    {y}}$ is continuous for every $    {x}\in     {\Omega}$ and $    {y}\in W^s_{loc}(    {x})$.
	\end{enumerate}
	\vspace{0.2cm}
	
	A \emph{unstable holonomy} for $    {F}_f$ is defined analogously for points in the same local unstable set. 
	
The domination condition in item (c) allows us to use the classical graph transform methods for obtaining stable laminations for the cocycle $    {F}_f$. This technique was developed in \cite{HPS}. Observe that $su$-domination is an open property, then there exists $\mathcal{W}$ an open neighborhood of ${f}$ in $\mathcal{H}_{1}(    {\Omega},\text{Lip}(S^1))$ such that $F_g$ is $su$-dominated for every $g\in \mathcal{W}$. In Proposition 5.1 and 5.2 of Avila and Viana \cite{AV}, the authors proved that in the present context the holonomies exist and vary continuously with the cocycle. This means that the map $f\mapsto h_{x,y}^{s,f}$ is continuous in $\mathcal{W}$ for every $x,y$ in the same stable set.

Let $g\in \mathcal{W}$. For every $F_g$-invariant probability measure $m$, denote by $m_x$ the Rokhlin disintegration into conditional probabilities associated to the partition $\{\{x\}\times S^1\}_{x\in\Omega}$. That is, $\{    {m}_{    {x}}\}_{     {x}\in     {\Omega}}$ is a measurable family of probability measures such that $    {m}_{    {x}}(\{    {x}\}\times S^1)=1$ for $    {\mu}$-almost every $    {x}\in    {\Omega}$ and $$    {m}(E)=\int     {m}_{    {x}}(E\cap (\{    {x}\}\times S^1))d    {\mu}$$ for every measurable set $E\subset     {\Omega}\times S^1$. See \cite{R}.

We say that a $F_g$-invariant probability measure $    {m}$ projecting to $\mu$ admits an s-invariant disintegration if there exists a $    {\mu}$-full measure set $E$ satisfying 
$$    {m}_{    {y}}=\left(h^{s,g}_{    {x},    {y}}\right)_*    {m}_{    {x}},$$
for $    {x},    {y}\in E$ and $    {y}\in W^s_{loc}(    {x})$. The measure $m$ is called an s-state if it admits an s-invariant disintegration. The definitions of u-invariant and u-state are analogous.

We say that $m$ is an su-state if it admits a disintegration which is  both s and u-invariant. 

Next we state two proposition that will allow us to prove Theorem A.

\begin{prop}\label{perturb} There exists $g$ arbitrarily close to $f$ such that $F_g$ does not admit $su$-states. 
\end{prop}

\begin{prop}\label{state} Let $m$ be a $F_g$-invariant probability measure projecting to $\mu$. If $\lambda_{con}(    {F}_g,    {m})=0$ and $\lambda_{con}(    {F}_g^{-1},    {m})=0$, then $m$ is an $su$-state.
\end{prop}

We explain how to conclude the proof of Theorem A from these propositions. Let $g$ be given by Proposition \ref{perturb}. We can prove that there exists an open set $\mathcal{V}\subset\mathcal{U}$ with $    {g}\in \mathcal{V}$ such that for every element $    {h}\in\mathcal{V}$, the cocycle $    {F}_{    {h}}$ does not admit $su$-states. This is done by contradiction. Suppose that the claim is not true, then we can find a sequence $    {g}_k$ converging to $    {g}$ such that every $    {F}_{    {g}_k}$ admits an su-state $    {m}_k$. Because of the compactness of the weak$^*$ topology, we know that there exists $    {m}$ a $    {F}_{    {g}}$-invariant probability measure and a sub-sequence of $    {m}_k$, that we continue denoting $    {m}_k$, such that the sequence $    {m}_k$ converges to $    {m}$ in the weak$^*$ topology. In Corollary 5.3 of \cite{AV} (see also \cite{TY}) it is stated that the limit of $su$-states is also an $su$-state. This contradicts the conclusion of Proposition \ref{perturb} for $F_g$. Therefore, there exists $\mathcal{V}$, neighborhood of $g$, with the desire property. Moreover, by Proposition \ref{state}, we have that for every $h\in \mathcal{V}$ and every $F_h$-invariant probability measure projecting to $\mu$, $$\lambda_{con}(  {F}_{  {h}},  {m})<0 \text{ or } \lambda_{con}( {{F}_  {h}}^{-1},  {m})<0.$$ This finished the proof of Theorem A. 
\qed

In the following, we present the proof of the propositions.
\vspace{0.2cm}

\noindent \textbf{Proposition \ref{perturb}.} \textit{In the hypotheses of Theorem A, there exists $g$ arbitrarily close to $f$ such that $F_g$ does not admit $su$-states. }
\vspace{0.1cm}

\begin{proof}

The pinching condition in item (b) guarantees that there exists a periodic point $    {x}_0$ such that $    {f}^{per(    {x}_0)}_{    {x}_0}$ has an attracting point $    {a}$ and a repelling point $    {r}$. For simplifying the notation we assume that $    {x}_0$ is a fixed point. 

Let $    {z}\in W^s(    {x}_0)\cap W^u(    {x}_0)$ be a homoclinic point of $    {x}_0$, and since $ W^s(    {x}_0)\cap W^u(    {x}_0)$ is dense,  we can assume that $    {z}$ is not in the same cylinder as $    {x}_0$. 

By definition, there exist two integers $k_1,k_2>0$ such that $    {\sigma}^{k_1}(    {z})\in W^s_{loc}(    {x}_0)$ and $    {\sigma}^{-k_2}(    {z})\in W^u_{loc}(    {x}_0)$. Therefore, since $su$-domination is an open condition, there exists an open set $\mathcal{W}$ such that for every $    {g}$ in $\mathcal{W}$, it is possible to define the maps $\eta_{i,    {g}}$, with $i=1,2$, as follows
$$\begin{aligned}
\eta_{1,     g}=&(    {g}^{k_1}_{    {z}})^{-1}\circ h^{s,    {g}}_{    {x}_0,    {\sigma}^{k_1}(    {z})},\\
\eta_{2,    {g}}=&     {g}^{k_2}_{    {\sigma}^{-k_2}(    {z})}\circ h^{u,    {g}}_{    {x}_0,    {\sigma}^{-k_2}(    {z})}.
\end{aligned}$$

Let $U$ denote an open neighborhood of $    {z}$ in $    {\Omega}$ such that $    {x}_0\not \in  U$ and $    {\sigma}^n(    {z})\not \in  U$ for every $n\neq 0$ and let $V$ be an open neighborhood of $    {z}$ such that $V\subset\overline{V}\subsetneq U$. Since $    {\Omega}$ is a compact metric space, there exists a Lipschitz bump function, $\phi\colon     {\Omega}\to\mathbb{R}$ such that $|\phi(    {x})|\leq1$ for every $     x\in      \Omega$, $\phi(    {x})=0$ in $U^c$ and $\phi(    {x})=1$ in $V$. 

Fix $\varepsilon>0$ such that the $d_{1}$-ball centered at $    {f}$ and with radius $\varepsilon$ is contained in $\mathcal{U}\cap \mathcal{W}$. Let $\delta\in\mathbb{R}$ satisfying the following properties: $0<\delta<{\varepsilon}/{(2H_1(\phi)) }$ and if $R_{\delta}$ is the rotation of angle $\delta$, then $$R_{\delta}(\{\eta_{1,    {f}}(    {a}),\eta_{1,    {f}}(    {r})\})\cap \{\eta_{2,    {f}}(    {a}),\eta_{2,    {f}}(    {r})\}=\emptyset.$$

We construct $g\in\mathcal{U}$ arbitrarily close to $f$ as follows: for every $x\in \Omega$, $$    {g}_{    {x}}=\phi(    {x})\cdot R_{\delta}\circ     {f}_{    {x}}+(1-\phi(    {x}))\cdot    {f}_{    {x}}=    {f}_{    {x}}+\phi(    {x})\delta=R_{\phi(    {x})\delta}\circ     {f}_{    {x}}.$$ In particular, $g_0=f_0$. 

Observe that $H_1(    {g}_{    {x}})=H_1(    {f}_{    {x}})$, this is a consequence of the following identities, $$d(    {g}_{    {x}}(    {p}),    {g}_{    {x}}(    {q}))=
d(R_{\phi(    {x})\delta}\circ     {f}_{    {x}}(    {p}),R_{\phi(    {x})\delta}\circ     {f}_{    {x}}(    {q}))=d(    {f}_{    {x}}(    {p}),    {f}_{    {x}}(    {q})),$$ 
since $R_{\phi(    {x})\delta}$ is an isometry. We have an analogous result for $     f^{-1}$ and $     g^{-1}$. Therefore, $$\begin{aligned}d_{\max}(    {g}_{    {x}},    {f}_{    {x}})&=\max\{\sup_{    {p}\in S^1}d(    {g}_{    {x}}(    {p}),    {f}_{    {x}}(    {p})),\sup_{    {p}\in S^1}d(    {g}^{-1}_{    {x}}(    {p}),    {f}^{-1}_{    {x}}(    {p}))\}\\ &\leq\sup_{    {x}\in     {\Omega}} \delta \phi(    {x})<{\varepsilon}/{2}.\end{aligned}$$

On the other hand, we need to compute ${H}_1({    {g}})$. For this it is enough to observe that
$$\begin{aligned} d_{\max}(    {g}_{    {x}},    {g}_{    {y}})&=\max\{
d_1(R_{\phi(    {x})\delta}\circ     {f}_{    {x}},R_{\phi(    {y})\delta}\circ    {f}_{    {y}}),d_1(    {f}_{    {x}}^{-1}\circ R_{-\phi(    {x})\delta},    {f}^{-1}_{    {y}}\circ R_{-\phi(    {y})\delta})\}\\
&\leq\left( {H}_1 (f)+\delta H_1(\phi)\right)d_{\Omega}(    {x},    {y}),
\end{aligned}$$
then, ${H}_1({    {g}})\leq {H}_1({    {f}})+\delta H_1(\phi)$. Exchanging the roles of $    {f}$ and $    {g}$, we obtain $$\left|{H}_1({    {g}})-{H}_1({    {f}})\right|\leq \varepsilon/2,$$ 
concluding that $d_{1}(    {g},    {f})<\varepsilon$ and then $    {g}$ is in the set $\mathcal{U}$.

Now we prove, by contradiction, that $F_g$ does not admit $su$-states. Thus, we suppose that $m$ is an su-state. By Proposition 4.8 of \cite{AV}, if $m$ is an su-state, then $     {m}$ admits a disintegration $\{     {m}_{     {x}}\}$ that is su-invariant for every $x\in \Omega$ and such that the function $x\mapsto m_x$ is continuous. In particular, it satisfies $$(     {g}_{     {x}_0})_*     {m}_{     {x}_0}=     {m}_{     {x}_0}.$$ This implies that $     {m}_{     {x}_0}$ is a $     {g}_{     {x}_0}$-invariant measure and it must be supported in the non-wandering set of $     {g}_{     {x}_0}$ which consists in the points $\{     {a},     {r}\}$.

Moreover, observe that for this construction we have the relations
\begin{equation}\label{phi}
	\eta_{1,     g}=R_\delta\circ \eta_{1,    {f}}\text{ and }\eta_{2,    {g}}=\eta_{2,    {f}}.
\end{equation} 
and the su-invariance implies $$\eta_{2,      g}^{-1}\circ \eta_{1,      g}(     {a})=     {a}\text{ or }\eta_{2,      g}^{-1}\circ \eta_{1,      g}(     {a})=     {r}.$$ Suppose $\eta_{2,      g}^{-1}\circ \eta_{1,      g}(     {a})=     {a}$, the other case is analogous. Then, $\eta_{1,      g}(     {a})=\eta_{2,      g}(     {a})$, and Equation (\ref{phi}) implies $$R_\delta \circ \eta_{1,      f}(     {a})=\eta_{2,      f}(     {a}),$$ which contradicts the choice of the $\delta$. Therefore, we have proved that there exists $g$ arbitrarily close to $f$ such that $     {F}_{     {g}}$ does not admit su-states. 
\end{proof}
\vspace{0.2cm}

\noindent \textbf{Proposition \ref{state}.} \textit{In the hypotheses of Theorem A, if $m$ is a $F_g$-invariant probability measure projecting to $\mu$ and $\lambda_{con}(    {F}_g,    {m})=0$ and $\lambda_{con}(    {F}_g^{-1},    {m})=0$, then $m$ is an $su$-state.  }
\vspace{0.1cm}

\begin{proof}

We are going to prove that if $\lambda_{con}(    {F}_g,    {m})=0$, then $    {m}$ is an s-state. Then, the proposition follows by applying the argument to both $    {F}_g$ and $    {{F}_g}^{-1}$.

We assume $\lambda_{con}(F_g,m)=0$. Recall that $    {F}_g$ is s-dominated, then we can define a cocycle $\tilde{F}_g\colon     {\Omega}\times S^1\to    {\Omega}\times S^1$ whose action along the fibers is constant on stable sets. This cocycle is defined by,
\begin{equation}\label{constant}
	\tilde{F}_{\tilde{g}}=\tilde{H}^{-1}\circ     {F}_g\circ \tilde{H},
\end{equation}
where $\tilde{H}:    {\Omega}\times S^1\to     {\Omega}\times S^1$  is given by  $$\tilde{H}(    {x},     {p})=(    {x},h^s_{\varphi(    {x}),    {x}}(    {p}))$$ and $\varphi\colon     {\Omega}\to    {\Omega}$ is  $\varphi(    {x})= W^s_{loc}(    {x})\cap W^u_{loc}(    {x}_i)$ where $x_i\in\left[ 0;i \right]$ is fixed for every $i=1,\ldots k$.

We denote as $\tilde{m}$ the $\tilde{F}_{\tilde{g}}$-invariant probability measure defined by,
\begin{equation}\label{mtilde}\tilde{m}_{    {x}}=(h^s_{    {x},\varphi(    {x})})_*    {m}_{    {x}},\end{equation}
for every $    {x}\in    {\Omega}$. 

The following result states the relation between the exponents of contraction of the cocycle $\tilde{F}_g$ with the original one $    {F}_g$. In particular, if $\lambda_{con}(    {F}_g,    {m})=0$, then $\lambda_{con}(\tilde{F}_g,\tilde{m})=0$. 

\begin{lemma}\label{zero}
	There exists a constant $R>0$ such that $$\lambda_{con}(    {F}_g,    {m})\leq R\; \lambda_{con}(\tilde{F}_g,\tilde{m})\leq 0.$$
\end{lemma}

\begin{proof}
	The fact that $\lambda_{con}(\tilde{F}_g,\tilde{m})\leq 0$ comes from the definition. We are left to prove the first inequality.

Observe that for every $x\in \Omega$ and $p\in S^1$, $${g}^n_{    {x}}(    {p})=h^s_{\varphi(    {\sigma}^n(    {x})),    {\sigma}^n(    {x})}\circ \tilde{g}^n_{    {x}}\circ h^s_{    {x},\varphi(    {x})}(    {p}),$$ where $\tilde{g}^n_{    {x}}$ denotes the action on the fiber of the cocycle $\tilde{F}_g$.

Then, for every $p,q \in S^1$, we want to bound the following distance,
	\begin{equation*}
		d(    {g}^n_{    {x}}(    {p}),    {g}^n_{    {x}}(    {q}))= \\ d(h^s_{\varphi(    {\sigma}^n(    {x})),    {\sigma}^n(    {x})}\circ \tilde{g}^n_{    {x}}\circ h^s_{    {x},\varphi(    {x})}(    {p}),h^s_{\varphi(    {\sigma}^n(    {x})),    {\sigma}^n(    {x})}\circ \tilde{g}^n_{    {x}}\circ h^s_{    {x},\varphi(    {x})}(    {q}))
	\end{equation*}
	with an expression depending only on $d(\tilde{g}^n_{    {x}}(    {p}),\tilde{g}^n_{    {x}}(    {q}))$. 
	
	Since $h^s_{    {x},\varphi(    {x})}$ is an homeomorphism, we can replace by $h^s_{    {x},\varphi(    {x})}    {p}=\tilde{p}$ and $h^s_{    {x},\varphi(    {x})}    {q}=\tilde{q}$, and the previous expression gets reduced to
	\begin{equation}
		\label{eq1}d(h^s_{\varphi(    {\sigma}^n(    {x})),    {\sigma}^n(    {x})}\circ \tilde{g}^n_{    {x}}(\tilde{p}),h^s_{\varphi(    {\sigma}^n(    {x})),    {\sigma}^n(    {x})}\circ \tilde{g}^n_{    {x}}(\tilde{q})).
	\end{equation}
	As $h^s_{\varphi(    {\sigma}^n(    {x})),    {\sigma}^n(    {x})}$ are $\gamma$-H\"older homeomorphism with uniform H\"older constant, we get that there exists $L>0$ such that
	$$d(h^s_{\varphi(    {\sigma}^n(    {x})),    {\sigma}^n(    {x})}\circ \tilde{g}^n_{    {x}}(\tilde{p}),h^s_{\varphi(    {\sigma}^n(    {x})),    {\sigma}^n(    {x})}\circ \tilde{g}^n_{    {x}}(\tilde{q}))\leq  Ld(\tilde{g}^n_{    {x}}(\tilde{p}),\tilde{g}^n_{    {x}}(\tilde{q}))^{\gamma}.$$
	
	Finally, we get the inequality
	$$\frac{\log (d(    {g}^n_{    {x}}(    {p}),    {g}^n_{    {x}}(    {q})))}{n}\leq \frac{\log (L) + \gamma\log(d(\tilde{g}^n_{    {x}}(\tilde{p}),\tilde{g}^n_{    {x}}(\tilde{q}))}{n},$$
	in which by making $n$ going to infinity and $    {q}$ approximating to $    {p}$, we obtain 
	$$\lambda_{con}(    {F}_g,    {x},    {p})\leq R\,\lambda_{con}(\tilde{F}_g,    {x},\tilde{p})=R\,\lambda_{con}(\tilde{F}_g,    {x},h^s_{    {x},\varphi(    {x})}    {p}),$$
	with $R=\gamma$. When we integrate with respect to the measure $    {m}$, we obtain 
\begin{equation*}
\begin{aligned}
\lambda_{con}(    {F}_g,    {m})&\leq R\int\lambda_{con}(\tilde{F}_g,    {x},h^s_{    {x},\varphi(    {x})}    {p})d    {m}\\
&=R\int\lambda_{con}(\tilde{F}_g,    {x},    {p})d\tilde{m}=R\, \lambda_{con}(\tilde{F}_g,\tilde{m}).
\end{aligned}
\end{equation*}
	
\end{proof}

Since $\tilde F_{\tilde g}$ is constant along $W^s_{loc}(x)$, it is possible to project it to a cocycle over the non-invertible shift. Recall that $P^s$ denotes the projection to the positive coordinates of $    {\Omega}$. Let $\hat\Omega=P^s(    {\Omega})$ and $\hat\sigma\colon\hat\Omega\to\hat\Omega$ be the unilateral shift satisfying $\hat\sigma\circ P^s=P^s\circ    {\sigma}$. If $    {\mu}$ is a $    {\sigma}$-invariant measure, then we consider $\hat\mu$ such that $P^s_*    {\mu}=\hat\mu$. Observe that $\hat\mu$ is a $\hat\sigma$-invariant measure.

Let $\hat F_{\hat g}\colon \hat\Omega\times S^1\to \hat\Omega\times S^1$ such that $\hat F_{\hat{g}}\circ (P^s,Id)=\tilde{F}_g$ where $\hat g$ is the map $\hat g\colon\hat\Omega\to \text{Lip}(S^1)$ satisfying $\hat{g}_{P^s(    {x})}=\tilde{g}_{    {x}}$. Thus, as we are assuming that the exponent of contraction of $\tilde F_{\tilde g}$ is zero, the Invariance Principle in \cite{M} concludes that $\hat m=(P^s\times Id)_*\tilde{m}$ is $\hat F_{\hat{g}}$-invariant, that is \begin{equation}\label{minv}
	\hat m_{\hat \sigma(x)}=\left(\hat g_{\hat x}\right)_*\hat m_{\hat x}.
\end{equation}
Finally, we recover the original measure $m$ together with the information that we obtained from $\hat m$ and the Invariance Principle. This is made by a well known result of convergence of martingale sequences that relates $\{\tilde{m}_{    {x}}\colon     {x}\in    {\Omega}\}$ and $\{\hat m_{\hat x}\colon \hat x\in\hat \Omega\}$ by
	\begin{equation}
	\label{igualdad}
	\tilde{m}_{    {x}}=\lim_{n\to\infty}\left(\hat{g}^n_{P^s(    {\sigma}^{-n}(    {x}))}\right)_*\hat m_{P^s(    {\sigma}^{-n}(    {x}))},
\end{equation}
at $    {\mu}$-almost every $    {x}\in    {\Omega}$, for details see \cite{V}. Using (\ref{minv}) in (\ref{igualdad}), we get that $\tilde{m}_{    {x}}=\hat m_{P^s(    {x})}=\tilde{m}_{    {z}}$ for $    {\mu}$ almost every $    {z}\in (P^s)^{-1}(x)=W^s_{loc}(    {x})$. Finally, 
$$ \left(h^s_{    {x},\varphi(    {x})}\right)_*    {m}_{    {x}}=\tilde{m}_{    {x}}=\tilde{m}_{    {z}}= \left(h^s_{    {z},\varphi(    {z})}\right)_*    {m}_{    {z}},$$
and since $\varphi(    {x})=\varphi(    {z})$, $$    {m}_{    {x}}=\left(h^s_{\varphi(    {x}),    {x}}\circ h^s_{    {z},\varphi(    {x})}\right)_*    {m}_{    {z}}=\left(h^s_{    {z},    {x}}\right)_*    {m}_{    {z}}$$
concluding the proof that $m$ is an s-state. 

\end{proof}

\subsection*{Acknowledgments} The authors thank Mauricio Poletti for suggesting this problem and the anonymous referee for useful observations that help improve the text. C.F. has was partially supported by FCT - Funda\c{c}\~ao para a Ci\^encia e a Tecnologia, under the project PTDC/MAT-PUR/29126/2017. K.M. has been supported by Conselho Nacional de Desenvolvimento Cient\'ifico e Tecnol\'ogico - CNPq - Brazil.

\end{document}